\documentclass[12pt]{amsart}
\usepackage{amsmath}
\usepackage{amsfonts}
\usepackage{amscd}
\usepackage{amsbsy}

\newtheorem{theorem}{Theorem}

\newtheorem{corollary}[theorem]{Corollary}

\newtheorem{definition}{Definition}
\newtheorem{example}{Example}

\newtheorem{lemma}[theorem]{Lemma}

\newtheorem{proposition}[theorem]{Proposition}
\newtheorem{remark}{Remark}

\begin{document}
\title{Ultraregular generalized functions }
\author{ Khaled BENMERIEM}
\address{University of Mascara, Algeria.}
\email{benmeriemkhaled@yahoo.fr}
\author{ Chikh BOUZAR}
\address{Oran-Essenia University, Algeria.}
\email{bouzar@yahoo.com}
\date{}
\subjclass[2000]{Primary 46F30, secondary 35A18, 46F10}
\keywords{Ultradifferentiable functions, Denjoy-Carleman classes, Product of
distributions, Generalized functions, Colombeau algebra, Wave front,
Microlocal analysis, Gevrey generalized functions}

\begin{abstract}
Algebras of ultradifferentiable generalized functions are introduced. We
give a microlocal analysis within these algebras related to the regularity
type and the ultradifferentiable property.
\end{abstract}

\maketitle

\section{Introduction}

The introduction by J.\ F. Colombeau of the algebra of generalized functions
$\mathcal{G}\left( \Omega \right) ,$ see \cite{Col}, containing the space of
distributions as a subspace and having the algebra of smooth functions as a
subalgebra, has initiated different directions of research in the field of
differential algebras of generalized functions, see \cite{AntRod}, \cite%
{GKOS}, \cite{NPS}, \cite{Ober} and \cite{benbou2}. \medskip

The current research of the regularity problem in algebras of generalized
functions of Colombeau type \ is based on the Oberguggenberger subalgebra $%
\mathcal{G}^{\infty }\left( \Omega \right) ,$ see \cite{Del}, \cite{HorKun},
\cite{HorObPil} and \cite{Marti}. This subalgebra plays the same role as $%
\mathcal{C}^{\infty }\left( \Omega \right) $\ in $\mathcal{D}^{\prime
}\left( \Omega \right) ,$\ and has indicated the importance of the
asymptotic behavior of the representative nets of a Colombeau generalized
function in studying regularity problems. However, the $\mathcal{G}^{\infty
}-$regularity does not exhaust the regularity questions inherent to the
Colombeau algebra $\mathcal{G}\left( \Omega \right) ,$ see \cite{Ober2}%
.\medskip

The purpose of this work is to introduce and to study new algebras of
generalized functions of Colombeau type, denoted by $\mathcal{G}^{M,\mathcal{%
R}}\left( \Omega \right) ,$ measuring regularity both by the asymptotic
behavior of the nets of smooth functions representing a Colombeau
generalized function and by their ultradifferentiable smoothness of
Denjoy-Carleman type $M=\left( M_{p}\right) _{p\in \mathbb{Z}_{+}}$.\
Elements of $\mathcal{G}^{M,\mathcal{R}}\left( \Omega \right) $\ are called
ultraregular generalized functions.\medskip

This paper is the composition of our two papers \cite{benbou3} and \cite%
{benbou1}.

\newpage

\section{Regular generalized functions}

Let $\Omega $ be a non void open subset of $\mathbb{R}^{n}$, define $%
\mathcal{X}\left( \Omega \right) $ as the space of elements $\left(
u_{\varepsilon }\right) _{\varepsilon }\ $of $C^{\infty }\left( \Omega
\right) ^{\left] 0,1\right] }$ such that, for every compact set $K\subset
\Omega $, $\forall \alpha \in \mathbb{Z}_{+}^{n},\exists m\in \mathbb{Z}_{+}$%
, $\exists C>0,\exists \eta \in \left] 0,1\right] ,\forall \varepsilon \in %
\left] 0,\eta \right] ,$%
\begin{equation*}
\sup\limits_{x\in K}\left\vert \partial ^{\alpha }u_{\varepsilon }\left(
x\right) \right\vert \leq C\varepsilon ^{-m}.
\end{equation*}%
By $\mathcal{N}\left( \Omega \right) $ we denote the elements $\left(
u_{\varepsilon }\right) _{\varepsilon }\in \mathcal{X}\left( \Omega \right) $
such that for every compact set $K\subset \Omega ,\forall \alpha \in \mathbb{%
Z}_{+}^{n},\forall m\in \mathbb{Z}_{+}$, $\exists C>0,\exists \eta \in \left]
0,1\right] ,\forall \varepsilon \in \left] 0,\eta \right] ,$%
\begin{equation*}
\sup\limits_{x\in K}\left\vert \partial ^{\alpha }u_{\varepsilon }\left(
x\right) \right\vert \leq C\varepsilon ^{m}.
\end{equation*}

\begin{definition}
The Colombeau algebra, denoted by $\mathcal{G}\left( \Omega \right) $, is
the quotient algebra
\begin{equation*}
\mathcal{G}\left( \Omega \right) =\frac{\mathcal{X}\left( \Omega \right) }{%
\mathcal{N}\left( \Omega \right) }.
\end{equation*}
\end{definition}

$\mathcal{G}\left( \Omega \right) $ is a commutative and associative
differential algebra containing $\mathcal{D}^{\prime }\left( \Omega \right) $
as a subspace and $C^{\infty }\left( \Omega \right) $ as a subalgebra. The
subalgebra of generalized functions with compact support, denoted $\mathcal{G%
}_{C}\left( \Omega \right) ,$ is the space of elements $f$ of $\mathcal{G}%
\left( \Omega \right) $ satisfying : there exist a representative $\left(
f_{\varepsilon }\right) _{\varepsilon \in \left] 0,1\right] }$\ of $f$\ and
a compact subset $K$ of $\Omega ,\forall \varepsilon \in \left] 0,1\right] ,$
$\mathrm{supp}f_{\varepsilon }\subset K.$

One defines the subalgebra of regular elements $\mathcal{G}^{\infty }\left(
\Omega \right) $, introduced by Oberguggenberger in \cite{Ober}, as the
quotient algebra
\begin{equation*}
\frac{\mathcal{X}^{\infty }\left( \Omega \right) }{\mathcal{N}\left( \Omega
\right) },
\end{equation*}%
where $\mathcal{X}^{\infty }\left( \Omega \right) $ is the space of elements
$\left( u_{\varepsilon }\right) _{\varepsilon }\ $of $C^{\infty }\left(
\Omega \right) ^{\left] 0,1\right] }$ such that, for every compact $K\subset
\Omega $, $\exists m\in \mathbb{Z}_{+},\forall \alpha \in \mathbb{Z}_{+}^{n}$%
, $\exists C>0,\exists \eta \in \left] 0,1\right] ,\forall \varepsilon \in %
\left] 0,\eta \right] ,$%
\begin{equation*}
\sup\limits_{x\in K}\left\vert \partial ^{\alpha }u_{\varepsilon }\left(
x\right) \right\vert \leq C\varepsilon ^{-m}.
\end{equation*}%
It is proved in \cite{Ober} the following fundamental result%
\begin{equation*}
\mathcal{G}^{\infty }\left( \Omega \right) \cap \mathcal{D}^{\prime }\left(
\Omega \right) =C^{\infty }\left( \Omega \right) .
\end{equation*}%
This means that the subalgebra $\mathcal{G}^{\infty }\left( \Omega \right) $%
\ plays in $\mathcal{G}\left( \Omega \right) $\ the same role as $C^{\infty
}\left( \Omega \right) $\ in $\mathcal{D}^{\prime }\left( \Omega \right) ,$
consequently one can introduce a local analysis by defining the generalized
singular support of $u\in \mathcal{G}\left( \Omega \right) .$\ This was the
first notion of regularity in Colombeau algebra.

Recently, different measures of regularity in algebras of generalized
functions have been proposed, see \cite{Del}, \cite{Marti} and \cite{Ober2}.
For our needs we recall the essential definitions and results on $\mathcal{R}
$-regular generalized functions, see \cite{Del}.

\begin{definition}
Let $\left( N_{m}\right) _{m\in \mathbb{Z}_{+}},\left( N_{m}^{\prime
}\right) _{m\in \mathbb{Z}_{+}}$ be two elements of $\mathbb{R}_{+}^{\mathbb{%
Z}_{+}}$, we write $\left( N_{m}\right) _{m\in \mathbb{Z}_{+}}\leq \left(
N_{m}^{\prime }\right) _{m\in \mathbb{Z}_{+}},$ if $\forall m\in \mathbb{Z}%
_{+}$, $N_{m}\leq N_{m}^{\prime }$. A non void subspace $\mathcal{R}$ of $%
\mathbb{R}_{+}^{\mathbb{Z}_{+}}$ is \ said regular, if the following
(R1)-(R3) are all satisfied :

For all $\left( N_{m}\right) _{m\in \mathbb{Z}_{+}}\in \mathcal{R}$ and $%
\left( k,k^{\prime }\right) \in \mathbb{Z}_{+}^{2},$ there exists $\left(
N_{m}^{\prime }\right) _{m\in \mathbb{Z}_{+}}\in \mathcal{R}$ such that
\begin{equation}
N_{m+k}+k^{\prime }\leq N_{m}^{\prime }\text{ , }\forall m\in \mathbb{Z}_{+}.
\tag{R1}  \label{c1}
\end{equation}

For all $\left( N_{m}\right) _{m\in \mathbb{Z}_{+}}$ and $\left(
N_{m}^{\prime }\right) _{m\in \mathbb{Z}_{+}}$ in $\mathcal{R}$ $,$ there
exists $\left( N"_{m}\right) _{m\in \mathbb{Z}_{+}}\in \mathcal{R}$ such
that
\begin{equation}
\max \left( N_{m},N_{m}^{\prime }\right) \leq N"_{m}\mathcal{\ }\text{, }%
\forall m\in \mathbb{Z}_{+}.  \tag{R2}  \label{c2}
\end{equation}

For all $\left( N_{m}\right) _{m\in \mathbb{Z}_{+}}$ and $\left(
N_{m}^{\prime }\right) _{m\in \mathbb{Z}_{+}}$ in $\mathcal{R}$ $,$ there
exists $\left( N"_{m}\right) _{m\in \mathbb{Z}_{+}}\in \mathcal{R}$ such
that
\begin{equation}
N_{l_{1}}+N_{l_{2}}^{\prime }\leq N"_{l_{1}+l_{2}}\text{ , }\forall \left(
l_{1},l_{2}\right) \in \mathbb{Z}_{+}^{2}.  \tag{R3}  \label{c3}
\end{equation}
\end{definition}

Define the $\mathcal{R}-$regular moderate elements of $\mathcal{X}\left(
\Omega \right) ,$ by
\begin{equation*}
\begin{array}{l}
\mathcal{X}^{\mathcal{R}}\left( \Omega \right) =\left\{ \underset{}{\left(
u_{\varepsilon }\right) _{\varepsilon }}\in \mathcal{X}\left( \Omega \right)
\mid \forall K\subset \subset \Omega ,\exists N\in \mathcal{R},\forall
\alpha \in \mathbb{Z}_{+}^{n},\right. \\
\multicolumn{1}{r}{\;\;\;\;\;\;\;\;\;\;\;\;\;\;\;\ \ \ \ \ \left. \exists
C>0,\exists \eta \in \left] 0,1\right] ,\forall \varepsilon \in \left]
0,\eta \right] :\sup\limits_{x\in K}\left\vert \partial ^{\alpha
}u_{\varepsilon }\left( x\right) \right\vert \leq C\varepsilon
^{-N_{\left\vert \alpha \right\vert }}\right\}}%
\end{array}%
\end{equation*}

and its ideal
\begin{equation*}
\begin{array}{l}
\mathcal{N}^{\mathcal{R}}\left( \Omega \right) =\left\{ \underset{}{\left(
u_{\varepsilon }\right) _{\varepsilon }}\in \mathcal{X}\left( \Omega \right)
\mid \forall K\subset \subset \Omega ,\forall N\in \mathcal{R},\forall
\alpha \in \mathbb{Z}_{+}^{n},\right. \\
\multicolumn{1}{r}{\;\;\;\;\;\;\;\;\;\;\;\;\;\;\;\ \ \ \ \ \left. \exists
C>0,\exists \eta \in \left] 0,1\right] ,\forall \varepsilon \in \left]
0,\eta \right] :\sup\limits_{x\in K}\left\vert \partial ^{\alpha
}u_{\varepsilon }\left( x\right) \right\vert \leq C\varepsilon
^{N_{\left\vert \alpha \right\vert }}\right\} .}%
\end{array}%
\end{equation*}

\begin{proposition}
1. The space $\mathcal{X}^{\mathcal{R}}\left( \Omega \right) $ is a
subalgebra of $\mathcal{X}\left( \Omega \right) $, stable by differentiation.

2. The set $\mathcal{N}^{\mathcal{R}}\left( \Omega \right) $ is an ideal of $%
\mathcal{X}^{\mathcal{R}}\left( \Omega \right) .$
\end{proposition}

\begin{remark}
From (\ref{c1}), one can show that $\mathcal{N}^{\mathcal{R}}\left( \Omega
\right) =\mathcal{N}\left( \Omega \right) .$
\end{remark}

\begin{definition}
The algebra of $\mathcal{R}$-regular generalized functions, denoted by $%
\mathcal{G}^{\mathcal{R}}\left( \Omega \right) ,$ is the quotient algebra
\begin{equation*}
\mathcal{G}^{\mathcal{R}}\left( \Omega \right) =\frac{\mathcal{X}^{\mathcal{R%
}}\left( \Omega \right) }{\mathcal{N}\left( \Omega \right) }.
\end{equation*}
\end{definition}

\begin{example}
The Colombeau algebra $\mathcal{G}\left( \Omega \right) $ is obtained when $%
\mathcal{R}=\mathbb{R}_{+}^{\mathbb{Z}_{+}},$ i.e. $\mathcal{G}^{\mathbb{R}%
_{+}^{\mathbb{Z}_{+}}}\left( \Omega \right) =\mathcal{G}\left( \Omega
\right) .$
\end{example}

\begin{example}
When
\begin{equation*}
\mathcal{A}=\left\{ \left( N_{m}\right) _{m\in \mathbb{Z}_{+}}\in \mathcal{R}%
:\exists a\geq 0,\exists b\geq 0,N_{m}\leq am+b,\forall m\in \mathbb{Z}%
_{+}\right\} ,
\end{equation*}
we obtain a differential subalgebra denoted by $\mathcal{G}^{\mathcal{A}%
}\left( \Omega \right) $.
\end{example}

\begin{example}
When $\mathcal{R}=\mathcal{B}$ the set of all bounded sequences, we obtain
the Oberguggenberger subalgebra, i.e. $\mathcal{G}^{\mathcal{B}}\left(
\Omega \right) =\mathcal{G}^{\infty }\left( \Omega \right) $.
\end{example}

\begin{example}
If we take $\mathcal{R}=\left\{ 0\right\} $, the condition (\ref{c1}) is not
hold, however we can define
\begin{equation*}
\mathcal{G}^{0}\left( \Omega \right) =\frac{\mathcal{X}^{\mathcal{R}}\left(
\Omega \right) }{\mathcal{N}\left( \Omega \right) }
\end{equation*}%
as the algebra of elements which have all derivatives locally uniformly
bounded for small $\varepsilon $, see \cite{Ober2}.
\end{example}

We have the following inclusions%
\begin{equation*}
\mathcal{G}^{0}\left( \Omega \right) \subset \mathcal{G}^{\mathcal{B}}\left(
\Omega \right) \subset \mathcal{G}^{\mathcal{A}}\left( \Omega \right)
\subset \mathcal{G}\left( \Omega \right) .
\end{equation*}

\section{Ultradifferentiable functions}

We recall some classical results on ultradifferentiable functions spaces. A
sequence of positive numbers $\left( M_{p}\right) _{p\in \mathbb{Z}_{+}}$ is
said to satisfy the following conditions:

$\left( H1\right) $ Logarithmic convexity, if
\begin{equation*}
M_{p}^{2}\leq M_{p-1}M_{p+1},\text{ }\forall p\geq 1.
\end{equation*}

$\left( H2\right) $ Stability under ultradifferential operators, if there
are constants $A>0$ and $H>0$ such that
\begin{equation*}
M_{p}\leq AH^{p}M_{q}M_{p-q},\forall p\geq q.
\end{equation*}

$\left( H3\right) ^{\prime }$ Non-quasi-analyticity, if
\begin{equation*}
\sum\limits_{p=1}^{\infty }\dfrac{M_{p-1}}{M_{p}}<+\infty .
\end{equation*}

\begin{remark}
Some results remain valid, see \cite{Kom}, when $\left( H2\right) $ is
replaced by the following weaker condition :

$\left( H2\right) ^{\prime }$ Stability under differential operators, if
there are constants $A>0$ and $H>0$ such that
\begin{equation*}
M_{p+1}\leq AH^{p}M_{p},\forall p\in \mathbb{Z}_{+}.
\end{equation*}
\end{remark}

The associated function of the sequence $\left( M_{p}\right) _{p\in \mathbb{Z%
}_{+}}$ is the function defined by
\begin{equation*}
\widetilde{M}\left( t\right) =\sup_{p}\ln \frac{t^{p}}{M_{p}},t\in \mathbb{R}%
_{+}^{\ast }.
\end{equation*}%
Some needed results of the associated function are given in the following
propositions proved in \cite{Kom}.

\begin{proposition}
\label{pro1}A positive sequence $\left( M_{p}\right) _{p\in \mathbb{Z}_{+}}$
satisfies condition $\left( H1\right) $ if and and only if
\begin{equation*}
M_{p}=M_{0}\sup_{t>0}\frac{t^{p}}{\exp \left( \widetilde{M}\left( t\right)
\right) },\text{ \ \ \ }p=0,1,...
\end{equation*}
\end{proposition}

\begin{proposition}
\label{pro2}A positive sequence $\left( M_{p}\right) _{p\in \mathbb{Z}_{+}}$
satisfies condition $\left( H2\right) $ if and and only if , $\exists
A>0,\exists H>0,\forall t>0,$%
\begin{equation}
2\widetilde{M}\left( t\right) \leq \widetilde{M}\left( Ht\right) +\ln \left(
AM_{0}\right) .  \label{(KomIneq)}
\end{equation}
\end{proposition}

\begin{remark}
We will always suppose that the sequence $\left( M_{p}\right) _{p\in \mathbb{%
Z}_{+}}$satisfies the condition $(H1)$\ and $M_{0}=1$.
\end{remark}

A differential operator of infinite order $P\left( D\right)
=\sum\limits_{\gamma \in \mathbb{Z}_{+}^{n}}a_{\gamma }D^{\gamma }$ is
called an ultradifferential operator of class $M=\left( M_{p}\right) _{p\in
\mathbb{Z}_{+}},$ if for every $h>0$ there exist a constant $c>0$ such that $%
\forall \gamma \in \mathbb{Z}_{+}^{n},$
\begin{equation}
\left\vert a_{\gamma }\right\vert \leq c\frac{h^{\left\vert \gamma
\right\vert }}{M_{\left\vert \gamma \right\vert }}.  \label{12}
\end{equation}

The class of ultradifferentiable functions of class $M$, denoted by $%
\mathcal{E}^{M}\left( \Omega \right) ,$ is the space of all $f\in C^{\infty
}\left( \Omega \right) $ satisfying for every compact subset $K$ of $\Omega $%
, $\exists c>0,\forall \alpha \in \mathbb{Z}_{+}^{n},$%
\begin{equation}
\sup_{x\in K}\left\vert \partial ^{\alpha }f\left( x\right) \right\vert \leq
c^{\left\vert \alpha \right\vert +1}M_{\left\vert \alpha \right\vert }.
\label{1}
\end{equation}%
This space is also called the space of Denjoy-Carleman$.$

\begin{example}
If $\left( M_{p}\right) _{p\in \mathbb{Z}_{+}}=\left( p!^{\sigma }\right)
_{p\in \mathbb{Z}_{+}},\sigma >1$, we obtain $\mathcal{E}^{\sigma }\left(
\Omega \right) $ the Gevrey space of order $\sigma $, and $\mathcal{A}\left(
\Omega \right) :=\mathcal{E}^{1}\left( \Omega \right) $ is the space of real
analytic functions defined on the open set $\Omega $.
\end{example}

The basic properties of the space $\mathcal{E}^{M}\left( \Omega \right) $\
are summarized in the following proposition, for the proof see \cite{Kom}.

\begin{proposition}
The space $\mathcal{E}^{M}\left( \Omega \right) $\ is an algebra. Moreover,
if $\left( M_{p}\right) _{p\in \mathbb{Z}_{+}}$ satisfies $\left( H2\right)
^{\prime },$ then $\mathcal{E}^{M}\left( \Omega \right) $ is stable by any
differential operator of finite order with coefficients in $\mathcal{E}%
^{M}\left( \Omega \right) $ and if $\left( M_{p}\right) _{p\in \mathbb{Z}%
_{+}}$ satisfies $\left( H2\right) $ then any ultradifferential operator of
class $M$ operates also as a sheaf homomorphism. The space $\mathcal{D}%
^{M}\left( \Omega \right) =\mathcal{E}^{M}\left( \Omega \right) \cap
\mathcal{D}\left( \Omega \right) $ is well defined and is not trivial if and
only if the sequence $\left( M_{p}\right) _{p\in \mathbb{Z}_{+}}$ satisfies $%
\left( H3\right) ^{\prime }.$
\end{proposition}

\begin{remark}
The strong dual of $\mathcal{D}^{M}\left( \Omega \right) $, denoted $%
\mathcal{D}^{\prime M}\left( \Omega \right) ,$ is called the space of
Roumieu ultraditributions.
\end{remark}

\section{Ultraregular generalized functions}

In the same way as $\mathcal{G}^{\infty }\left( \Omega \right) ,$ $\mathcal{G%
}^{\mathcal{R}}\left( \Omega \right) $\ forms a sheaf of differential
subalgebras of $\mathcal{G}\left( \Omega \right) $,\ consequently one
defines the generalized\textit{\ }$\mathcal{R}-$singular support of $u\in
\mathcal{G}\left( \Omega \right) $, denoted by \textrm{singsupp}$_{\mathcal{R%
}}$\thinspace $u$, as the complement in $\Omega $ of the largest set $\Omega
^{\prime }$ such that $u_{/\Omega ^{\prime }}\in \mathcal{G}^{\mathcal{R}%
}\left( \Omega ^{\prime }\right) ,$ where $u_{/\Omega ^{\prime }}\ $means
the restriction of the generalized function $u$ on $\Omega ^{\prime }.$ This
new notion of regularity is linked with the asymptotic limited growth of
generalized functions. Our aim in this section is to introduce a more
precise notion of regularity within the Colombeau algebra taking into
account both the asymptotic growth and the smoothness property of
generalized functions. We introduce general algebras of ultradifferentiable $%
\mathcal{R}-$regular generalized functions of class $M$, where the sequence $%
M=(M_{p})_{p\in \mathbb{Z}_{+}}$ satisfies the conditions $\left( H1\right) $
with $M_{0}=1,\left( H2\right) $ and $\left( H3\right) ^{\prime }.$

\begin{definition}
The space of ultraregular moderate elements of class $M$, denoted $\mathcal{X%
}^{M,\mathcal{R}}\left( \Omega \right) ,$ is the space of $\left(
f_{\varepsilon }\right) _{\varepsilon }\in C^{\infty }\left( \Omega \right)
^{\left] 0,1\right] }$ satisfying for every compact $K$ of $\Omega $, $%
\exists N\in \mathcal{R},\exists C>0,\exists \varepsilon _{0}\in \left] 0,1%
\right] ,\forall \alpha \in \mathbb{Z}_{+}^{n},\forall \varepsilon \leq
\varepsilon _{0},$
\begin{equation}
\sup_{x\in K}\left\vert \partial ^{\alpha }f_{\varepsilon }\left( x\right)
\right\vert \leq C^{\left\vert \alpha \right\vert +1}M_{\left\vert \alpha
\right\vert }\varepsilon ^{-N_{\left\vert \alpha \right\vert }}.  \label{2}
\end{equation}%
The space of null elements is defined as $\mathcal{N}^{M,\mathcal{R}}\left(
\Omega \right) :=\mathcal{N}\left( \Omega \right) \cap \mathcal{X}^{M,%
\mathcal{R}}\left( \Omega \right) .$
\end{definition}

The main properties of the spaces $\mathcal{X}^{M,\mathcal{R}}\left( \Omega
\right) $ and $\mathcal{N}^{M,\mathcal{R}}\left( \Omega \right) $ are given
in the following proposition.

\begin{proposition}
\label{pro3}1) The space $\mathcal{X}^{M,\mathcal{R}}\left( \Omega \right) $
is a subalgebra of $\mathcal{X}\left( \Omega \right) $ stable by action of
differential operators

2) The space $\mathcal{N}^{M,\mathcal{R}}\left( \Omega \right) $ is an ideal
of $\mathcal{X}^{M,\mathcal{R}}\left( \Omega \right) .$
\end{proposition}

\begin{proof}
1) Let $\left( f_{\varepsilon }\right) _{\varepsilon },\left( g_{\varepsilon
}\right) _{\varepsilon }\in \mathcal{X}^{M,\mathcal{R}}\left( \Omega \right)
$ and $K$ a compact subset of $\Omega $, then

$\exists N\in \mathcal{R},\exists C_{1}>0,\exists \varepsilon _{1}\in \left]
0,1\right] ,$ such that $\forall \beta \in \mathbb{Z}_{+}^{n},\forall x\in
K,\forall \varepsilon \leq \varepsilon _{1},$%
\begin{equation}
\left\vert \partial ^{\beta }f_{\varepsilon }\left( x\right) \right\vert
\leq C_{1}^{\left\vert \beta \right\vert +1}M_{\left\vert \beta \right\vert
}\varepsilon ^{-N_{\left\vert \beta \right\vert }},  \label{4}
\end{equation}

$\exists N^{\prime }\in \mathcal{R},\exists C_{2}>0,\exists \varepsilon
_{2}\in \left] 0,1\right] ,$ such that $\forall \beta \in \mathbb{Z}%
_{+}^{n},\forall x\in K,\forall \varepsilon \leq \varepsilon _{2},$%
\begin{equation}
\left\vert \partial ^{\beta }g_{\varepsilon }\left( x\right) \right\vert
\leq C_{2}^{\left\vert \beta \right\vert +1}M_{\left\vert \beta \right\vert
}\varepsilon ^{-N_{\left\vert \beta \right\vert }^{\prime }}.  \label{5}
\end{equation}%
It clear from (\ref{c2}) that $\left( f_{\varepsilon }+g_{\varepsilon
}\right) _{\varepsilon }\in \mathcal{X}^{M,\mathcal{R}}\left( \Omega \right)
.$ Let $\alpha \in \mathbb{Z}_{+}^{n},$ then
\begin{equation*}
\left\vert \partial ^{\alpha }\left( f_{\varepsilon }g_{\varepsilon }\right)
\left( x\right) \right\vert \leq \sum_{\beta =0}^{\alpha }\binom{\alpha }{%
\beta }\left\vert \partial ^{\alpha -\beta }f_{\varepsilon }\left( x\right)
\right\vert \left\vert \partial ^{\beta }g_{\varepsilon }\left( x\right)
\right\vert .
\end{equation*}%
From (\ref{c3}) $\exists N"\in \mathcal{R}$ such that, $\forall \beta \leq
\alpha ,$ $N_{\left\vert \alpha -\beta \right\vert }+N_{\left\vert \beta
\right\vert }^{\prime }\leq N"_{\left\vert \alpha \right\vert },$ and from $%
\left( H1\right) $, we have $M_{p}M_{q}\leq M_{p+q}$, then for $\varepsilon
\leq \min \left\{ \varepsilon _{1},\varepsilon _{2}\right\} $ and $x\in K,$
we have
\begin{eqnarray*}
\frac{\varepsilon ^{N"_{\left\vert \alpha \right\vert }}}{M_{\left\vert
\alpha \right\vert }}\left\vert \partial ^{\alpha }\left( f_{\varepsilon
}g_{\varepsilon }\right) \left( x\right) \right\vert &\leq &\sum_{\beta
=0}^{\alpha }\binom{\alpha }{\beta }\frac{\varepsilon ^{N_{\left\vert \alpha
-\beta \right\vert }}}{M_{\left\vert \alpha -\beta \right\vert }}\left\vert
\partial ^{\alpha -\beta }f_{\varepsilon }\left( x\right) \right\vert \times
\\
&&\times \frac{\varepsilon ^{N_{\left\vert \beta \right\vert }^{\prime }}}{%
M_{\left\vert \beta \right\vert }}\left\vert \partial ^{\beta
}g_{\varepsilon }\left( x\right) \right\vert \\
&\leq &\sum_{\beta =0}^{\alpha }\binom{\alpha }{\beta }C_{1}^{\left\vert
\alpha -\beta \right\vert +1}C_{2}^{\left\vert \beta \right\vert +1} \\
&\leq &C^{\left\vert \alpha \right\vert +1},
\end{eqnarray*}%
where $C=\max \left\{ C_{1}C_{2},C_{1}+C_{2}\right\} $, then $\left(
f_{\varepsilon }g_{\varepsilon }\right) _{\varepsilon }\in \mathcal{X}^{M,%
\mathcal{R}}\left( \Omega \right) $.

Let now $\alpha ,\beta \in \mathbb{Z}_{+}^{n},$ where $\left\vert \beta
\right\vert =1$, then for $\varepsilon \leq \varepsilon _{1}$ and $x\in K,$
we have
\begin{equation*}
\left\vert \partial ^{\alpha }\left( \partial ^{\beta }f_{\varepsilon
}\right) \left( x\right) \right\vert \leq C_{1}^{\left\vert \alpha
\right\vert +2}M_{\left\vert \alpha \right\vert +1}\varepsilon
^{-N_{\left\vert \alpha \right\vert +1}}.
\end{equation*}%
From (\ref{c1}), $\exists N^{\prime }\in \mathcal{R}$, such that $%
N_{\left\vert \alpha \right\vert +1}\leq N_{\left\vert \alpha \right\vert
}^{\prime }$, and from $\left( H2\right) ^{\prime },\exists A>0,H>0$, such
that $M_{\left\vert \alpha \right\vert +1}\leq AH^{\left\vert \alpha
\right\vert }M_{\left\vert \alpha \right\vert }$, we have
\begin{eqnarray*}
\left\vert \partial ^{\alpha }\left( \partial ^{\beta }f_{\varepsilon
}\right) \left( x\right) \right\vert &\leq &AC_{1}^{2}\left( C_{1}H\right)
^{\left\vert \alpha \right\vert }M_{\left\vert \alpha \right\vert
}\varepsilon ^{-N_{\left\vert \alpha \right\vert }^{\prime }} \\
&\leq &C^{\left\vert \alpha \right\vert +1}M_{\left\vert \alpha \right\vert
}\varepsilon ^{-N_{\left\vert \alpha \right\vert }^{\prime }},
\end{eqnarray*}%
which means $\left( \partial ^{\beta }f_{\varepsilon }\right) _{\varepsilon
}\in \mathcal{X}^{M,\mathcal{R}}\left( \Omega \right) .$

2) The facts that $\mathcal{N}^{M,\mathcal{R}}\left( \Omega \right) =%
\mathcal{N}\left( \Omega \right) \cap \mathcal{X}^{M,\mathcal{R}}\left(
\Omega \right) \subset \mathcal{X}^{M,\mathcal{R}}\left( \Omega \right) $
and $\mathcal{N}\left( \Omega \right) =\mathcal{N}^{\mathcal{R}}\left(
\Omega \right) $ is an ideal of $\mathcal{X}^{\mathcal{R}}\left( \Omega
\right) $ give that $\mathcal{N}^{M,\mathcal{R}}\left( \Omega \right) $ is
an ideal of $\mathcal{X}^{M,\mathcal{R}}\left( \Omega \right) .$
\end{proof}

The following definition introduces the algebra of ultraregular generalized
functions.

\begin{definition}
The algebra of ultraregular generalized functions of class $M=\left(
M_{p}\right) _{p\in \mathbb{Z}_{+}},$ denoted $\mathcal{G}^{M,\mathcal{R}%
}\left( \Omega \right) ,$ is the quotient algebra
\begin{equation}
\mathcal{G}^{M,\mathcal{R}}\left( \Omega \right) =\frac{\mathcal{X}^{M,%
\mathcal{R}}\left( \Omega \right) }{\mathcal{N}^{M,\mathcal{R}}\left( \Omega
\right) }.
\end{equation}
\end{definition}

The basic properties of $\mathcal{G}^{M,\mathcal{R}}\left( \Omega \right) $\
are given in the following assertion.

\begin{proposition}
$\mathcal{G}^{M,\mathcal{R}}\left( \Omega \right) $ is a differential
subalgebra of $\mathcal{G}\left( \Omega \right) .$
\end{proposition}

\begin{proof}
The algebraic properties hold from proposition \ref{pro3}.
\end{proof}

\begin{example}
If we take the set $\mathcal{R}=\mathcal{B}$ we obtain as a particular case
the algebra $\mathcal{G}^{M,\mathcal{B}}\left( \Omega \right) $ of \cite%
{Marti} denoted there by $\mathcal{G}^{L}\left( \Omega \right) .\ $
\end{example}

\begin{example}
If we take $\left( M_{p}\right) _{p\in \mathbb{Z}_{+}}=\left( p!^{\sigma
}\right) _{p\in \mathbb{Z}_{+}}$ we obtain a new subalgebra $\mathcal{G}%
^{\sigma ,\mathcal{R}}\left( \Omega \right) $ of $\mathcal{G}\left( \Omega
\right) $\ called the algebra of Gevrey regular generalized functions of
order $\sigma .$
\end{example}

\begin{example}
If we take both the set $\mathcal{R}=\mathcal{B}$ and $\left( M_{p}\right)
_{p\in \mathbb{Z}_{+}}=\left( p!^{\sigma }\right) _{p\in \mathbb{Z}_{+}}$ we
obtain a new algebra, denoted $\mathcal{G}^{\sigma ,\infty }\left( \Omega
\right) ,$ that we will call the Gevrey-Oberguggenberger algebra of order $%
\sigma .$
\end{example}

\begin{remark}
In \cite{benbou2}\ is introduced an algebra of generalized Gevrey
ultradistributions containing the classical Gevrey space $\mathcal{E}%
^{\sigma }\left( \Omega \right) $\ as a subalgebra and the space of Gevrey
ultradistributions $\mathcal{D}_{3\sigma -1}^{\prime }\left( \Omega \right) $%
\ as a subspace.
\end{remark}

It is not evident how to obtain, without more conditions, that $\mathcal{X}%
^{M,\mathcal{R}}\left( \Omega \right) $ is stable by action of
ultradifferential operators of class $M,$\ however we have the following
result.

\begin{proposition}
Suppose that the regular set $\mathcal{R}$\ satisfies as well the following
condition : For all $\left( N_{k}\right) _{k\in \mathbb{Z}_{+}}\in \mathcal{R%
}$ $,$ there exist an $\left( N_{k}^{\ast }\right) _{k\in \mathbb{Z}_{+}}\in
\mathcal{R},$ and positive numbers$\ h>0,L>0,\forall m\in \mathbb{Z}%
_{+},\forall \varepsilon \in \left] 0,1\right] ,$
\begin{equation}
\sum\limits_{k\in \mathbb{Z}_{+}}h^{k}\varepsilon ^{-N_{k+m}}\leq
L\varepsilon ^{-N_{m}^{\ast }}.  \label{(R4)}
\end{equation}%
\ Then the algebra $\mathcal{X}^{M,\mathcal{R}}\left( \Omega \right) $ is
stable by action of ultradifferential operators of class $M.$
\end{proposition}

\begin{proof}
Let $\left( f_{\varepsilon }\right) _{\varepsilon }\in \mathcal{X}^{M,%
\mathcal{R}}\left( \Omega \right) $ and $P\left( D\right)
=\sum\limits_{\beta \in \mathbb{Z}_{+}^{n}}a_{\beta }D^{\beta }$ be an
ultradifferential operator of class $M,$ then for any compact set $K$ of $%
\Omega ,$ $\exists \left( N_{m}\right) _{m\in \mathbb{Z}_{+}}\in \mathcal{R}%
,\exists C>0,\exists \varepsilon _{0}\in \left] 0,1\right] ,$ such that $%
\forall \alpha \in \mathbb{Z}_{+}^{n},\forall x\in K,\forall \varepsilon
\leq \varepsilon _{1},$%
\begin{equation*}
\left\vert \partial ^{\alpha }f_{\varepsilon }\left( x\right) \right\vert
\leq C^{\left\vert \alpha \right\vert +1}M_{\left\vert \alpha \right\vert
}\varepsilon ^{-N_{\left\vert \alpha \right\vert }}.
\end{equation*}%
For every $h>0$ there exists a $c>0$ such that $\forall \beta \in \mathbb{Z}%
_{+}^{n},$
\begin{equation*}
\left\vert a_{\beta }\right\vert \leq c\frac{h^{\left\vert \beta \right\vert
}}{M_{\left\vert \beta \right\vert }}.
\end{equation*}%
Let $\alpha \in \mathbb{Z}_{+}^{n}$, then
\begin{eqnarray*}
\frac{h^{\left\vert \alpha \right\vert }}{M_{\left\vert \alpha \right\vert }}%
\left\vert \partial ^{\alpha }\left( P\left( D\right) f_{\varepsilon
}\right) \left( x\right) \right\vert &\leq &\sum\limits_{\beta \in \mathbb{Z}%
_{+}^{n}}\left\vert a_{\beta }\right\vert \frac{h^{\left\vert \alpha
\right\vert }}{M_{\left\vert \alpha \right\vert }}\left\vert \partial
^{\alpha +\beta }f_{\varepsilon }\left( x\right) \right\vert \\
&\leq &c\sum\limits_{\beta \in \mathbb{Z}_{+}^{n}}\frac{h^{\left\vert \beta
\right\vert }}{M_{\left\vert \beta \right\vert }}\frac{h^{\left\vert \alpha
\right\vert }}{M_{\left\vert \alpha \right\vert }}\left\vert \partial
^{\alpha +\beta }f_{\varepsilon }\left( x\right) \right\vert .
\end{eqnarray*}%
From $\left( H2\right) $ and $\left( f_{\varepsilon }\right) _{\varepsilon
}\in \mathcal{X}^{M,\mathcal{R}}\left( \Omega \right) ,$ then$\ \exists
C>0,\exists A>0,\exists H>0,$%
\begin{equation*}
\frac{h^{\left\vert \alpha \right\vert }}{M_{\left\vert \alpha \right\vert }}%
\left\vert \partial ^{\alpha }\left( P\left( D\right) f_{\varepsilon
}\right) \left( x\right) \right\vert \leq A^{\left\vert \alpha \right\vert
}\sum\limits_{\beta \in \mathbb{Z}_{+}^{n}}h^{\left\vert \beta \right\vert
}\varepsilon ^{-N_{\left\vert \alpha +\beta \right\vert }},
\end{equation*}%
consequently by condition (\ref{(R4)}), there exist $\left( N_{k}^{\ast
}\right) _{k\in \mathbb{Z}_{+}}\in \mathcal{R},h>0,L>0,\forall \varepsilon
\in \left] 0,1\right] ,$%
\begin{equation*}
\frac{h^{\left\vert \alpha \right\vert }}{M_{\left\vert \alpha \right\vert }}%
\left\vert \partial ^{\alpha }\left( P\left( D\right) f_{\varepsilon
}\right) \left( x\right) \right\vert \leq A^{\left\vert \alpha \right\vert
}L\varepsilon ^{-N_{\left\vert \alpha \right\vert }^{\ast }}
\end{equation*}%
\ which shows that $\left( P\left( D\right) f_{\varepsilon }\right)
_{\varepsilon }\in \mathcal{X}^{M,\mathcal{R}}\left( \Omega \right) .$
\end{proof}

\begin{example}
The sets $\left\{ 0\right\} $ and $\mathcal{B}$ satisfy the condition (\ref%
{(R4)}).
\end{example}

The space $\mathcal{E}^{M}\left( \Omega \right) $ is embedded into $\mathcal{%
G}^{M,\mathcal{R}}\left( \Omega \right) $ for all $\mathcal{R}$ by the
canonical map
\begin{equation*}
\begin{array}{ccc}
\sigma :\mathcal{E}^{M}\left( \Omega \right) & \rightarrow & \mathcal{G}^{M,%
\mathcal{R}}\left( \Omega \right) \\
u & \rightarrow & \left[ u_{\varepsilon }\right]%
\end{array}%
,
\end{equation*}%
where $u_{\varepsilon }=u$ for all $\varepsilon \in \left] 0,1\right] ,$
which is an injective homomorphism of algebras.

\begin{proposition}
The following diagram%
\begin{equation*}
\begin{array}{ccccc}
\mathcal{E}^{M}\left( \Omega \right) & \rightarrow & C^{\infty }\left(
\Omega \right) & \rightarrow & \mathcal{D}^{\prime }\left( \Omega \right) \\
\downarrow &  & \downarrow &  & \downarrow \\
\mathcal{G}^{M,\mathcal{B}}\left( \Omega \right) & \rightarrow & \mathcal{G}%
^{\mathcal{B}}\left( \Omega \right) & \rightarrow & \mathcal{G}\left( \Omega
\right)%
\end{array}%
\end{equation*}%
\textit{is commutative.}
\end{proposition}

\begin{proof}
The embeddings in the diagram are canonical except the embedding $\mathcal{D}%
^{\prime }\left( \Omega \right) \rightarrow \mathcal{G}\left( \Omega \right)
,$\ which is now well known in framework of Colombeau generalized functions,
see \cite{GKOS} for details. The commutativity of the diagram is then
obtained easily from the commutativity of the classical diagram
\begin{equation*}
\begin{array}{ccc}
C^{\infty }\left( \Omega \right) & \rightarrow & \mathcal{D}^{\prime }\left(
\Omega \right) \\
& \searrow & \downarrow \\
&  & \mathcal{G}\left( \Omega \right)%
\end{array}%
\end{equation*}
\end{proof}

A fundamental result of regularity in $\mathcal{G}\left( \Omega \right) $ is
the following.

\begin{theorem}
We have $\mathcal{G}^{M,\mathcal{B}}\left( \Omega \right) \cap \mathcal{D}%
^{\prime }\left( \Omega \right) =\mathcal{E}^{M}\left( \Omega \right) .$
\end{theorem}

\begin{proof}
Let $u=cl\left( u_{\varepsilon }\right) _{\varepsilon }\in \mathcal{G}^{M,%
\mathcal{B}}\left( \Omega \right) \cap C^{\infty }\left( \Omega \right) ,$
i.e. $\left( u_{\varepsilon }\right) _{\varepsilon }\in \mathcal{X}^{M,%
\mathcal{B}}\left( \Omega \right) $, then we have for every compact set $%
K\subset \Omega ,\exists N\in \mathbb{Z}_{+},\exists c>0,\exists \eta \in %
\left] 0,1\right] ,$%
\begin{equation*}
\forall \alpha \in \mathbb{Z}_{+}^{n},\forall \varepsilon \in \left] 0,\eta %
\right] :\sup\limits_{x\in K}\left\vert \partial ^{\alpha }u\left( x\right)
\right\vert \leq c^{\left\vert \alpha \right\vert +1}M_{\left\vert \alpha
\right\vert }\varepsilon ^{-N}.
\end{equation*}%
When choosing $\varepsilon =\eta $, we obtain
\begin{equation*}
\forall \alpha \in \mathbb{Z}_{+}^{n},\sup\limits_{x\in K}\left\vert
\partial ^{\alpha }u\left( x\right) \right\vert \leq c^{\left\vert \alpha
\right\vert +1}M_{\left\vert \alpha \right\vert }\eta ^{-N}\leq
c_{1}^{\left\vert \alpha \right\vert +1}M_{\left\vert \alpha \right\vert },
\end{equation*}%
where $c_{1}$ depends only on $K$. Then $u$ is in $\mathcal{E}^{M}\left(
\Omega \right) $. This shows that $\mathcal{G}^{M,\mathcal{B}}\left( \Omega
\right) \cap C^{\infty }\left( \Omega \right) \subset \mathcal{E}^{M}\left(
\Omega \right) $. As the reverse inclusion is obvious, then we have proved $%
\mathcal{G}^{M,\mathcal{B}}\left( \Omega \right) \cap C^{\infty }\left(
\Omega \right) =\mathcal{E}^{M}\left( \Omega \right) $. Consequently
\begin{eqnarray*}
\mathcal{G}^{M,\mathcal{B}}\left( \Omega \right) \cap \mathcal{D}^{\prime
}\left( \Omega \right) &=&\left( \mathcal{G}^{M,\mathcal{B}}\left( \Omega
\right) \cap \mathcal{G}^{\mathcal{B}}\left( \Omega \right) \right) \cap
\mathcal{D}^{\prime }\left( \Omega \right) \\
&=&\mathcal{G}^{M,\mathcal{B}}\left( \Omega \right) \cap \left( \mathcal{G}^{%
\mathcal{B}}\left( \Omega \right) \cap \mathcal{D}^{\prime }\left( \Omega
\right) \right) \\
&=&\mathcal{G}^{M,\mathcal{B}}\left( \Omega \right) \cap C^{\infty }\left(
\Omega \right) \\
&=&\mathcal{E}^{M}\left( \Omega \right)
\end{eqnarray*}
\end{proof}

\begin{proposition}
The algebra $\mathcal{G}^{M,\mathcal{R}}\left( \Omega \right) $ is a sheaf
of subalgebras of $\mathcal{G}\left( \Omega \right) .$
\end{proposition}

\begin{proof}
The sheaf property of $\mathcal{G}^{M,\mathcal{R}}\left( \Omega \right) $\
is obtained in the same way as the sheaf properties of $\mathcal{G}^{%
\mathcal{R}}\left( \Omega \right) $\ and $\mathcal{E}^{M}\left( \Omega
\right) $.
\end{proof}

We can now give a new tool of $\mathcal{G}^{M,\mathcal{R}}$-local regularity
analysis.

\begin{definition}
Define the $(M,\mathcal{R})$-singular support of a generalized function $%
u\in \mathcal{G}\left( \Omega \right) ,$ denoted by $\mathrm{singsupp}_{M,%
\mathcal{R}}\left( u\right) ,$ as the complement of the largest open set $%
\Omega ^{\prime }$ such that $u\in \mathcal{G}^{M,\mathcal{R}}\left( \Omega
^{\prime }\right) .$
\end{definition}

The basic property of $\mathrm{singsupp}_{M,\mathcal{R}}$\ is summarized in
the following proposition, which is easy to prove by the facts above.

\begin{proposition}
\label{pseudlocProp}Let $P\left( x,D\right) =\sum\limits_{\left\vert \alpha
\right\vert \leq m}a_{\alpha }\left( x\right) D^{\alpha }$ be a generalized
linear partial differential operator with $\mathcal{G}^{M,\mathcal{R}}\left(
\Omega \right) $ coefficients, then
\begin{equation}
\mathrm{singsupp}_{M,\mathcal{R}}\left( P\left( x,D\right) u\right) \subset
\mathrm{singsupp}_{M,\mathcal{R}}\left( u\right) ,\forall u\in \mathcal{G}%
\left( \Omega \right)  \label{pseudloc}
\end{equation}
\end{proposition}

We can now introduce a local generalized analysis in the sense of Colombeau
algebra. Indeed, a generalized linear partial differential operator with $%
\mathcal{G}^{M,\mathcal{R}}\left( \Omega \right) $ coefficients $P\left(
x,D\right) $\ is said $(M,\mathcal{R})-$hypoelliptic$\ $in $\Omega ,$\ if%
\begin{equation}
\mathrm{singsupp}_{M,\mathcal{R}}\left( P\left( x,D\right) u\right) =\mathrm{%
singsupp}_{M,\mathcal{R}}\left( u\right) ,\forall u\in \mathcal{G}\left(
\Omega \right)  \label{Hyp}
\end{equation}%
Such a problem in this general form is still in the beginning. Of course, a
microlocalization of the problem (\ref{Hyp})\ will lead to a more precise
information about solutions of generalized linear partial differential
equations. A first attempt is done in the following section.

\section{Affine ultraregular generalized functions}

Although we have defined a tool for a local $(M,\mathcal{R})-$analysis\ in $%
\mathcal{G}\left( \Omega \right) $, it is not clear how to microlocalize
this concept in general. We can do it in the general situation\ of affine
ultraregularity. This is the aim of this section.

\begin{definition}
Define the affine regular sequences
\begin{equation*}
\mathcal{A}=\left\{ \left( N_{m}\right) _{m\in \mathbb{Z}_{+}}\in \mathcal{R}%
:\exists a\geq 0,\exists b\geq 0,N_{m}\leq am+b,\forall m\in \mathbb{Z}%
_{+}\right\} .
\end{equation*}
\end{definition}

A basic $(M,\mathcal{A})-$microlocal analysis\ in $\mathcal{G}\left( \Omega
\right) $\ can be developed due to the following result.

\begin{proposition}
\label{ref4}Let $f=cl\left( f_{\varepsilon }\right) _{\varepsilon }\in
\mathcal{G}_{C}\left( \Omega \right) ,$ then $f$ is $\mathcal{A}-$%
ultraregular of class $M=\left( M_{p}\right) _{p\in \mathbb{Z}_{+}}$ if and
only if $\exists a\geq 0,\exists b\geq 0,\exists C>0,\exists k>0,$ $\exists
\varepsilon _{0}\in \left] 0,1\right] ,\forall \varepsilon \leq \varepsilon
_{0},$ such that
\begin{equation}
\left\vert \mathcal{F}\left( f_{\varepsilon }\right) \left( \xi \right)
\right\vert \leq C\varepsilon ^{-b}\exp \left( -\widetilde{M}\left(
k\varepsilon ^{a}\left\vert \xi \right\vert \right) \right) ,\forall \xi \in
\mathbb{R}^{n},  \label{3-2}
\end{equation}%
where $\mathcal{F}$ denotes the Fourier transform.
\end{proposition}

\begin{proof}
Suppose that $f=cl\left( f_{\varepsilon }\right) _{\varepsilon }\in \mathcal{%
G}_{C}\left( \Omega \right) \cap \mathcal{G}^{M,\mathcal{A}}\left( \Omega
\right) ,$ then $\exists K$ compact of $\Omega ,\exists C>0,\exists N\in
\mathcal{A},\exists \varepsilon _{1}>0,\forall \alpha \in \mathbb{Z}%
_{+}^{n},\forall x\in K,\forall \varepsilon \leq \varepsilon _{0},$ $\mathrm{%
supp}f_{\varepsilon }\subset K,$ such that
\begin{equation}
\left\vert \partial ^{\alpha }f_{\varepsilon }\right\vert \leq C^{\left\vert
\alpha \right\vert +1}M_{\left\vert \alpha \right\vert }\varepsilon
^{-N_{\left\vert \alpha \right\vert }},
\end{equation}%
so we have, $\forall \alpha \in \mathbb{Z}_{+}^{n},$
\begin{equation*}
\left\vert \xi ^{\alpha }\right\vert \left\vert \mathcal{F}\left(
f_{\varepsilon }\right) \left( \xi \right) \right\vert \leq \left\vert \int
\exp \left( -ix\xi \right) \partial ^{\alpha }f_{\varepsilon }\left(
x\right) dx\right\vert
\end{equation*}%
then, $\exists C>0,\forall \varepsilon \leq \varepsilon _{0},$%
\begin{equation*}
\left\vert \xi \right\vert ^{\left\vert \alpha \right\vert }\left\vert
\mathcal{F}\left( f_{\varepsilon }\right) \left( \xi \right) \right\vert
\leq C^{\left\vert \alpha \right\vert +1}M_{\left\vert \alpha \right\vert
}\varepsilon ^{-N_{\left\vert \alpha \right\vert }}.
\end{equation*}%
Therefore%
\begin{eqnarray*}
\left\vert \mathcal{F}\left( f_{\varepsilon }\right) \left( \xi \right)
\right\vert &\leq &C\inf_{\alpha }\left\{ \frac{C^{\left\vert \alpha
\right\vert }M_{\left\vert \alpha \right\vert }}{\left\vert \xi \right\vert
^{\left\vert \alpha \right\vert }}\varepsilon ^{-N_{\left\vert \alpha
\right\vert }}\right\} \\
&\leq &C\varepsilon ^{-b}\inf_{\alpha }\left\{ \left( \frac{\varepsilon
^{-a}C}{\left\vert \xi \right\vert }\right) ^{\left\vert \alpha \right\vert
}M_{\left\vert \alpha \right\vert }\right\} \\
&\leq &C\varepsilon ^{-b}\exp \left( -\widetilde{M}\left( \frac{\varepsilon
^{a}\left\vert \xi \right\vert }{\sqrt{n}C}\right) \right) .
\end{eqnarray*}%
Hence $\exists C>0,\exists k>0,$%
\begin{equation*}
\left\vert \mathcal{F}\left( f_{\varepsilon }\right) \left( \xi \right)
\right\vert \leq C\varepsilon ^{-b}\exp \left( -\widetilde{M}\left(
k\varepsilon ^{a}\left\vert \xi \right\vert \right) \right) \text{ ,}
\end{equation*}%
i.e. we have (\ref{3-2}).

Suppose now that (\ref{3-2}) is valid, then from inequality (\ref{(KomIneq)}%
) of the Proposition \ref{pro2}, $\exists C,C^{\prime },C^{\prime \prime
}>0,\exists \varepsilon _{0}\in \left] 0,1\right] ,\forall \varepsilon \leq
\varepsilon _{0},$
\begin{eqnarray*}
\left\vert \partial ^{\alpha }f_{\varepsilon }\left( x\right) \right\vert
&\leq &C\varepsilon ^{-b}\sup_{\xi }\left\vert \xi \right\vert ^{\left\vert
\alpha \right\vert }\exp \left( -\widetilde{M}\left( \frac{k}{H}\varepsilon
^{a}\left\vert \xi \right\vert \right) \right) \times \\
&&\times \int \exp \left( -\widetilde{M}\left( \frac{k}{H}\varepsilon
^{a}\left\vert \xi \right\vert \right) \right) d\xi \\
&\leq &C^{\prime }\varepsilon ^{-a\left\vert \alpha \right\vert -b}\sup_{\xi
}\left\vert \frac{k}{H}\varepsilon ^{a}\xi \right\vert ^{\left\vert \alpha
\right\vert }\exp \left( -\widetilde{M}\left( \frac{k}{H}\varepsilon
^{a}\left\vert \xi \right\vert \right) \right) \\
&\leq &C^{\prime \prime }\varepsilon ^{-a\left\vert \alpha \right\vert
-b}\sup_{\eta }\left\vert \eta \right\vert ^{\left\vert \alpha \right\vert
}\exp \left( -\widetilde{M}\left( \left\vert \eta \right\vert \right)
\right) .
\end{eqnarray*}%
Due to the Proposition \ref{pro1}$,$ then $\exists C>0,\exists N\in \mathcal{%
A},$ such that
\begin{equation*}
\left\vert \partial ^{\alpha }f_{\varepsilon }\left( x\right) \right\vert
\leq C^{\left\vert \alpha \right\vert +1}M_{\left\vert \alpha \right\vert
}\varepsilon ^{-N_{\left\vert \alpha \right\vert }},\text{ }
\end{equation*}%
where $C=\max \left( C,\frac{1}{k}\right) $, and $N_{m}=am+b$, then $f\in
\mathcal{G}^{M,\mathcal{A}}\left( \Omega \right) .$
\end{proof}

\begin{corollary}
Let $f=cl\left( f_{\varepsilon }\right) _{\varepsilon }\in \mathcal{G}%
_{C}\left( \Omega \right) ,$ then $f$ is a Gevrey affine\ ultraregular
generalized function of order $\sigma $, i.e. $f\in \mathcal{G}^{\sigma ,%
\mathcal{A}}\left( \Omega \right) $, if and only if $\exists a\geq 0,\exists
b\geq 0,\exists C>0,\exists k>0,$ $\exists \varepsilon _{0}>0,\forall
\varepsilon \leq \varepsilon _{0},$ such that
\begin{equation}
\left\vert \mathcal{F}\left( f_{\varepsilon }\right) \left( \xi \right)
\right\vert \leq C\varepsilon ^{-b}\exp \left( -k\varepsilon ^{a}\left\vert
\xi \right\vert ^{\frac{1}{\sigma }}\right) ,\forall \xi \in \mathbb{R}^{n}.
\end{equation}%
In particular, $f$ is a Gevrey generalized function of order $\sigma $, i.e.
$f\in \mathcal{G}^{\sigma ,\infty }\left( \Omega \right) $, if and only if $%
\exists b\geq 0,\exists C>0,\exists k>0,$ $\exists \varepsilon
_{0}>0,\forall \varepsilon \leq \varepsilon _{0},$ such that
\begin{equation}
\left\vert \mathcal{F}\left( f_{\varepsilon }\right) \left( \xi \right)
\right\vert \leq C\varepsilon ^{-b}\exp \left( -k\left\vert \xi \right\vert
^{\frac{1}{\sigma }}\right) ,\forall \xi \in \mathbb{R}^{n}.
\end{equation}
\end{corollary}

Using the above results, we can define the concept of $\mathcal{G}^{M,%
\mathcal{A}}-$wave front of $u\in \mathcal{G}\left( \Omega \right) $ and
give the basic elements of\ a $(M,\mathcal{A})-$generalized microlocal
analysis within the Colombeau algebra $\mathcal{G}\left( \Omega \right) $.

\begin{definition}
Define the cone $\sum_{\mathcal{A}}^{M}\left( f\right) \subset \mathbb{R}%
^{n}\backslash \left\{ 0\right\} ,f\in \mathcal{G}_{C}\left( \Omega \right) $%
, as the complement of the set of points having a conic neighborhood $\Gamma
$ such that $\exists a\geq 0,\exists b\geq 0,\exists C>0,\exists k>0,$ $%
\exists \varepsilon _{0}>0,\forall \varepsilon \leq \varepsilon _{0},$ such
that
\begin{equation}
\left\vert \mathcal{F}\left( f_{\varepsilon }\right) \left( \xi \right)
\right\vert \leq C\varepsilon ^{-b}\exp \left( -\widetilde{M}\left(
k\varepsilon ^{a}\left\vert \xi \right\vert \right) \right) ,\forall \xi \in
\mathbb{R}^{n}.  \label{3-4}
\end{equation}
\end{definition}

\begin{proposition}
\label{ref5-1}For every $f\in \mathcal{G}_{C}\left( \Omega \right) ,$ we have

1. The set $\sum_{\mathcal{A}}^{M}\left( f\right) $ is a closed subset.

2. $\sum_{\mathcal{A}}^{M}\left( f\right) =\emptyset \Longleftrightarrow
f\in \mathcal{G}^{M,\mathcal{A}}\left( \Omega \right) .$
\end{proposition}

\begin{proof}
The proof of 1. is clear from the definition, and 2. holds from Proposition %
\ref{ref4}.
\end{proof}

\begin{proposition}
\label{ref5} For every $f\in \mathcal{G}_{C}\left( \Omega \right) ,$ we have
\begin{equation*}
\sum_{\mathcal{A}}^{M}\left( \psi f\right) \subset \sum_{\mathcal{A}%
}^{M}\left( f\right) ,\forall \psi \in \mathcal{E}^{M}\left( \Omega \right) .
\end{equation*}
\end{proposition}

\begin{proof}
Let $\xi _{0}\notin \sum_{\mathcal{A}}^{M}\left( f\right) $, i.e. $\exists
\Gamma $ a conic neighborhood of $\xi _{0},\exists a\geq 0,\exists
b>0,\exists k_{1}>0,\exists c_{1}>0,\exists \varepsilon _{1}\in \left] 0,1%
\right] ,$ such that $\forall \xi \in \Gamma ,\forall \varepsilon \leq
\varepsilon _{1},$
\begin{equation}
\left\vert \mathcal{F}\left( f_{\varepsilon }\right) \left( \xi \right)
\right\vert \leq c_{1}\varepsilon ^{-b}\exp \left( -\widetilde{M}\left(
k_{1}\varepsilon ^{a}\left\vert \xi \right\vert \right) \right) ,
\label{3-5}
\end{equation}%
Let $\chi \in \mathcal{D}^{M}\left( \Omega \right) ,$ $\chi =1$ on
neighborhood of $\mathrm{supp}f$, then $\chi \psi \in \mathcal{D}^{M}\left(
\Omega \right) ,\forall \psi \in \mathcal{E}^{M}\left( \Omega \right) $,
hence, see \cite{Kom}, $\exists k_{2}>0,\exists c_{2}>0,\forall \xi \in
\mathbb{R}^{n},$
\begin{equation}
\left\vert \mathcal{F}\left( \chi \psi \right) \left( \xi \right)
\right\vert \leq c_{2}\exp \left( -\widetilde{M}\left( k_{2}\left\vert \xi
\right\vert \right) \right) .  \label{3-6}
\end{equation}%
Let $\Lambda $ be a conic neighborhood of $\xi _{0}$ such that, $\overline{%
\Lambda }\subset \Gamma ,$ then we have, for $\xi \in \Lambda ,$
\begin{eqnarray*}
\mathcal{F}\left( \chi \psi f_{\varepsilon }\right) \left( \xi \right)
&=&\int_{\mathbb{R}^{n}}\mathcal{F}\left( f_{\varepsilon }\right) \left(
\eta \right) \mathcal{F}\left( \chi \psi \right) \left( \xi -\eta \right)
d\eta \\
&=&\int_{A}\mathcal{F}\left( f_{\varepsilon }\right) \left( \eta \right)
\mathcal{F}\left( \chi \psi \right) \left( \xi -\eta \right) d\eta + \\
&&+\int_{B}\mathcal{F}\left( f_{\varepsilon }\right) \left( \eta \right)
\mathcal{F}\left( \chi \psi \right) \left( \xi -\eta \right) d\eta ,
\end{eqnarray*}%
where $A=\left\{ \eta ;\left\vert \xi -\eta \right\vert \leq \delta \left(
\left\vert \xi \right\vert +\left\vert \eta \right\vert \right) \right\} $
and $B=\left\{ \eta ;\left\vert \xi -\eta \right\vert >\delta \left(
\left\vert \xi \right\vert +\left\vert \eta \right\vert \right) \right\} $.
Take $\delta $ sufficiently small such that $\eta \in \Gamma ,\frac{%
\left\vert \xi \right\vert }{2}<\left\vert \eta \right\vert <2\left\vert \xi
\right\vert ,\forall \in \Lambda ,\forall \eta \in A$, then $\exists
c>0,\forall \varepsilon \leq \varepsilon _{1},$%
\begin{eqnarray*}
\left\vert \int_{A}\mathcal{F}\left( f_{\varepsilon }\right) \left( \eta
\right) \mathcal{F}\left( \chi \psi \right) \left( \xi -\eta \right) d\eta
\right\vert &\leq &c\varepsilon ^{-b}\exp \left( -\widetilde{M}\left(
k_{1}\varepsilon ^{a}\frac{\left\vert \xi \right\vert }{2}\right) \right)
\times \\
&&\times \int_{A}\exp \left( -\widetilde{M}\left( k_{2}\left\vert \xi -\eta
\right\vert \right) \right) d\eta ,
\end{eqnarray*}%
so $\exists c>0,\exists k>0,$
\begin{equation}
\left\vert \int_{A}\mathcal{F}\left( f_{\varepsilon }\right) \left( \eta
\right) \mathcal{F}\left( \chi \psi \right) \left( \xi -\eta \right) d\eta
\right\vert \leq c\varepsilon ^{-b}\exp \left( -\widetilde{M}\left(
k\varepsilon ^{a}\left\vert \xi \right\vert \right) \right) .  \label{3-9}
\end{equation}%
As $f\in \mathcal{G}_{C}\left( \Omega \right) ,$ then $\exists q\in \mathbb{Z%
}_{+},\exists m>0,\exists c>0,\exists \varepsilon _{2}>0,\forall \varepsilon
\leq \varepsilon _{2},$
\begin{equation*}
\left\vert \mathcal{F}\left( f_{\varepsilon }\right) \left( \xi \right)
\right\vert \leq c\varepsilon ^{-q}\left\vert \xi \right\vert ^{m}.
\end{equation*}%
Hence for $\varepsilon \leq \min \left( \varepsilon _{1},\varepsilon
_{2}\right) ,$ $\exists c>0$, such that we have
\begin{eqnarray*}
\left\vert \int_{B}\mathcal{F}\left( f_{\varepsilon }\right) \left( \eta
\right) \mathcal{F}\left( \chi \psi \right) \left( \xi -\eta \right) d\eta
\right\vert &\leq &c\varepsilon ^{-q}\int_{B}\left\vert \eta \right\vert
^{m}\exp \left( -\widetilde{M}\left( k_{2}\left\vert \xi -\eta \right\vert
\right) \right) d\eta \\
&\leq &c\varepsilon ^{-q}\int_{B}\left\vert \eta \right\vert ^{m}\exp \left(
-\widetilde{M}\left( k_{2}\delta \left( \left\vert \xi \right\vert
+\left\vert \eta \right\vert \right) \right) \right) d\eta .
\end{eqnarray*}%
We have, from Proposition \ref{pro2}, i.e. inequality (\ref{(KomIneq)}), $%
\exists H>0,\exists A>0,\forall t_{1}>0,\forall t_{2}>0,$
\begin{equation}
-\widetilde{M}\left( t_{1}+t_{2}\right) \leq -\widetilde{M}\left( \frac{t_{1}%
}{H}\right) -\widetilde{M}\left( \frac{t_{2}}{H}\right) +\ln A\text{ ,}
\end{equation}%
so
\begin{eqnarray*}
\left\vert \int_{B}\mathcal{F}\left( f_{\varepsilon }\right) \left( \eta
\right) \mathcal{F}\left( \chi \psi \right) \left( \xi -\eta \right) d\eta
\right\vert &\leq &cA\varepsilon ^{-q}\exp \left( -\widetilde{M}\left( \frac{%
k_{2}\delta }{H}\left\vert \xi \right\vert \right) \right) \times \\
&&\times \int_{B}\left\vert \eta \right\vert ^{m}\exp \left( -\widetilde{M}%
\left( \frac{k_{2}\delta }{H}\left\vert \eta \right\vert \right) \right)
d\eta .
\end{eqnarray*}%
Hence $\exists c>0,\exists k>0,$ such that%
\begin{equation}
\left\vert \int_{B}\widehat{f_{\varepsilon }}\left( \eta \right) \widehat{%
\psi }\left( \xi -\eta \right) d\eta \right\vert \leq c\varepsilon ^{-q}\exp
\left( -\widetilde{M}\left( k\varepsilon ^{a}\left\vert \xi \right\vert
\right) \right) .  \label{3-10}
\end{equation}%
Consequently, (\ref{3-9}) and (\ref{3-10}) give $\xi _{0}\notin \sum_{%
\mathcal{A}}^{M}\left( \psi f\right) .$
\end{proof}

We define $\sum_{\mathcal{A}}^{M}\left( f\right) $ of a generalized function
$f$ at a point $x_{0}$ and the affine wave front set of class $M$ in $%
\mathcal{G}\left( \Omega \right) .$

\begin{definition}
Let $f\in \mathcal{G}\left( \Omega \right) $ and $x_{0}\in \Omega $, the
cone of affine singular directions of class $M=(M_{p})$ of $f$ at $x_{0}$,
denoted by $\sum_{\mathcal{A},x_{0}}^{M}\left( f\right) $, is
\begin{equation}
\sum\nolimits_{\mathcal{A},x_{0}}^{M}\left( f\right) =\bigcap \left\{
\sum\nolimits_{\mathcal{A}}^{M}\left( \phi f\right) :\phi \in \mathcal{D}%
^{M}\left( \Omega \right) ,\phi =1\text{ on a neighborhood of }x_{0}\right\}
.  \label{3-11}
\end{equation}
\end{definition}

The following lemma gives the relation between the local and microlocal $(M,%
\mathcal{A})-$analysis in $\mathcal{G}\left( \Omega \right) $.

\begin{lemma}
\label{lem1}Let $f\in \mathcal{G}\left( \Omega \right) $, then
\begin{equation*}
\sum\nolimits_{\mathcal{A},x_{0}}^{M}\left( f\right) =\emptyset
\Longleftrightarrow x_{0}\notin \mathrm{singsupp}_{M,\mathcal{A}}\left(
f\right) .
\end{equation*}
\end{lemma}

\begin{proof}
See the proof of the analogical Lemma in \cite{benbou2}.
\end{proof}

\begin{definition}
A point $\left( x_{0},\xi _{0}\right) \notin WF_{\mathcal{A}}^{M}\left(
f\right) \subset \Omega \times \mathbb{R}^{n}\backslash \left\{ 0\right\} ,$
if there exist $\phi \in \mathcal{D}^{M}\left( \Omega \right) ,\phi \equiv 1$
on a neighborhood of $x_{0}$, a conic neighborhood $\Gamma $ of $\xi _{0}$,
and numbers $a\geq 0,b\geq 0,k>0,c>0,\varepsilon _{0}\in \left] 0,1\right] ,$
such that $\forall \varepsilon \leq \varepsilon _{0},\forall \xi \in \Gamma
, $
\begin{equation*}
\left\vert \mathcal{F}\left( \phi f_{\varepsilon }\right) \left( \xi \right)
\right\vert \leq c\varepsilon ^{-b}\exp \left( -\widetilde{M}\left(
k\varepsilon ^{a}\left\vert \xi \right\vert \right) \right) .
\end{equation*}
\end{definition}

\begin{remark}
A point $\left( x_{0},\xi _{0}\right) \notin WF_{\mathcal{A}}^{M}\left(
f\right) \subset \Omega \times \mathbb{R}^{n}\backslash \left\{ 0\right\} $
means $\xi _{0}\notin \sum\nolimits_{\mathcal{A},x_{0}}^{M}\left( f\right) .$
\end{remark}

The basic properties of $WF_{\mathcal{A}}^{M}$ are given in the following
proposition.

\begin{proposition}
Let $f\in \mathcal{G}\left( \Omega \right) $, then

1) The projection of $WF_{\mathcal{A}}^{M}\left( f\right) $ on $\Omega $ is
the $\mathrm{singsupp}_{M,\mathcal{A}}\left( f\right) .$

2) If $f\in \mathcal{G}_{C}\left( \Omega \right) ,$ then the projection of $%
WF_{\mathcal{A}}^{M}\left( f\right) $ on $\mathbb{R}^{n}\backslash \left\{
0\right\} $ is $\sum_{\mathcal{A}}^{M}\left( f\right) .$

3) $WF_{\mathcal{A}}^{M}\left( \partial ^{\alpha }f\right) \subset WF_{%
\mathcal{A}}^{M}\left( f\right) ,\forall \alpha \in \mathbb{Z}_{+}^{n}.$

4) $WF_{\mathcal{A}}^{M}\left( gf\right) \subset WF_{\mathcal{A}}^{M}\left(
f\right) ,\forall g\in \mathcal{G}^{M,\mathcal{A}}\left( \Omega \right) .$
\end{proposition}

\begin{proof}
1) and 2) holds from the definition, Proposition \ref{ref5-1} and Lemma \ref%
{lem1}.

3) Let $\left( x_{0},\xi _{0}\right) \notin WF_{\mathcal{A}}^{M}\left(
f\right) $. Then $\exists \phi \in \mathcal{D}^{M}\left( \Omega \right)
,\phi \equiv 1$ on $\overline{U}$, where $U$ is a neighborhood of $x_{0}$,
there exist a conic neighborhood $\Gamma $ of $\xi _{0},$ and $\exists a\geq
0,\exists b>0,\exists k_{2}>0,\exists c_{1}>0,\exists \varepsilon _{0}\in %
\left] 0,1\right] ,$ such that $\forall \xi \in \Gamma ,\forall \varepsilon
\leq \varepsilon _{0},$
\begin{equation}
\left\vert \mathcal{F}\left( \phi f_{\varepsilon }\right) \left( \xi \right)
\right\vert \leq c_{1}\varepsilon ^{-b}\exp \left( -\widetilde{M}\left(
k_{2}\varepsilon ^{a}\left\vert \xi \right\vert \right) \right) .
\label{3-13}
\end{equation}%
We have, for $\psi \in \mathcal{D}^{M}\left( U\right) $ such that $\psi
\left( x_{0}\right) =1$,
\begin{eqnarray*}
\left\vert \mathcal{F}\left( \psi \partial f_{\varepsilon }\right) \left(
\xi \right) \right\vert &=&\left\vert \mathcal{F}\left( \partial \left( \psi
f_{\varepsilon }\right) \right) \left( \xi \right) -\mathcal{F}\left( \left(
\partial \psi \right) f_{\varepsilon }\right) \left( \xi \right) \right\vert
\\
&\leq &\left\vert \xi \right\vert \left\vert \mathcal{F}\left( \psi \phi
f_{\varepsilon }\right) \left( \xi \right) \right\vert +\left\vert \mathcal{F%
}\left( \left( \partial \psi \right) \phi f_{\varepsilon }\right) \left( \xi
\right) \right\vert .
\end{eqnarray*}%
As $WF_{\mathcal{A}}^{M}\left( \psi f\right) \subset WF_{\mathcal{A}%
}^{M}\left( f\right) ,$ so (\ref{3-13}) holds for both $\left\vert \mathcal{F%
}\left( \psi \phi f_{\varepsilon }\right) \left( \xi \right) \right\vert \ $%
and $\left\vert \mathcal{F}\left( \left( \partial \psi \right) \phi
f_{\varepsilon }\right) \left( \xi \right) \right\vert .$ Then
\begin{eqnarray*}
\left\vert \xi \right\vert \left\vert \mathcal{F}\left( \psi \phi
f_{\varepsilon }\right) \left( \xi \right) \right\vert &\leq &c\varepsilon
^{-b}\left\vert \xi \right\vert \exp \left( -\widetilde{M}\left(
k_{2}\varepsilon ^{a}\left\vert \xi \right\vert \right) \right) \\
&\leq &c^{\prime }\varepsilon ^{-b-a}\exp \left( -\widetilde{M}\left(
k_{3}\varepsilon ^{a}\left\vert \xi \right\vert \right) \right) ,
\end{eqnarray*}%
with some $c^{\prime }>0,k_{3}>0,(k_{3}<k_{2}),$ such that
\begin{equation*}
\varepsilon ^{a}\left\vert \xi \right\vert \leq c^{\prime }\exp \left(
\widetilde{M}\left( k_{2}\varepsilon ^{a}\left\vert \xi \right\vert \right) -%
\widetilde{M}\left( k_{3}\varepsilon ^{a}\left\vert \xi \right\vert \right)
\right)
\end{equation*}
for $\varepsilon $ sufficiently small. Hence (\ref{3-13}) holds for $%
\left\vert \mathcal{F}\left( \psi \partial f_{\varepsilon }\right) \left(
\xi \right) \right\vert ,$ which proves $\left( x_{0},\xi _{0}\right) \notin
WF_{\mathcal{A}}^{M}\left( \partial f\right) .$

4) Let $\left( x_{0},\xi _{0}\right) \notin WF_{\mathcal{A}}^{M}\left(
f\right) $. Then there exist $\phi \in \mathcal{D}^{M}\left( \Omega \right)
,\phi \left( x\right) =1$ on a neighborhood $U$ of $x_{0}$, a conic
neighborhood $\Gamma $ of $\xi _{0},$ and numbers $a_{1}\geq
0,b_{1}>0,k_{1}>0,c_{1}>0,\varepsilon _{1}\in \left] 0,1\right] ,$ such that
$\forall \varepsilon \leq \varepsilon _{1},\forall \xi \in \Gamma ,$
\begin{equation*}
\left\vert \mathcal{F}\left( \phi f_{\varepsilon }\right) \left( \xi \right)
\right\vert \leq c_{1}\varepsilon ^{-b_{1}}\exp \left( -\widetilde{M}\left(
k_{1}\varepsilon ^{a_{1}}\left\vert \xi \right\vert \right) \right) .
\end{equation*}%
Let $\psi \in \mathcal{D}^{M}\left( \Omega \right) $ and $\psi =1$ on $%
\mathrm{supp}\phi ,$ then $\mathcal{F}\left( \phi g_{\varepsilon
}f_{\varepsilon }\right) =\mathcal{F}\left( \psi g_{\varepsilon }\right)
\ast \mathcal{F}\left( \phi f_{\varepsilon }\right) .$ We have $\psi g\in
\mathcal{G}^{M,\mathcal{A}}\left( \Omega \right) ,$ then $\exists c_{2}>0,$ $%
\exists a_{2}\geq 0,\exists b_{2}>0,\exists k_{2}>0,\exists \varepsilon
_{2}>0,\forall \xi \in \mathbb{R}^{n},\forall \varepsilon \leq \varepsilon
_{2},$
\begin{equation}
\left\vert \mathcal{F}\left( \psi g_{\varepsilon }\right) \left( \xi \right)
\right\vert \leq c_{2}\varepsilon ^{-b_{2}}\exp \left( -\widetilde{M}\left(
k_{2}\varepsilon ^{a_{2}}\left\vert \xi \right\vert \right) \right) .
\end{equation}%
We have
\begin{eqnarray*}
\mathcal{F}\left( \phi g_{\varepsilon }f_{\varepsilon }\right) \left( \xi
\right) &=&\int_{A}\mathcal{F}\left( \phi f_{\varepsilon }\right) \left(
\eta \right) \mathcal{F}\left( \psi g_{\varepsilon }\right) \left( \xi -\eta
\right) d\eta + \\
&&+\int_{B}\mathcal{F}\left( \phi f_{\varepsilon }\right) \left( \eta
\right) \mathcal{F}\left( \psi g_{\varepsilon }\right) \left( \xi -\eta
\right) d\eta ,
\end{eqnarray*}%
where $A$ and $B$ are the same as in the proof of Proposition \ref{ref5}. By
the same reasoning we obtain the proof.
\end{proof}

A micolocalization of Proposition \ref{pseudlocProp}\ is expressed in the
following result.

\begin{corollary}
Let $P\left( x,D\right) =\sum\limits_{\left\vert \alpha \right\vert \leq
m}a_{\alpha }\left( x\right) D^{\alpha }$ be a generalized linear partial
differential operator with $\mathcal{G}^{M,\mathcal{A}}\left( \Omega \right)
$ coefficients, then
\begin{equation*}
WF_{\mathcal{A}}^{M}\left( P\left( x,D\right) f\right) \subset WF_{\mathcal{A%
}}^{M}\left( f\right) ,\forall f\in \mathcal{G}\left( \Omega \right)
\end{equation*}
\end{corollary}

\begin{remark}
The reverse inclusion will give a generalized microlocal ultraregularity of
linear partial differential operator with coefficients in $\mathcal{G}^{M,%
\mathcal{A}}\left( \Omega \right) $. The first case of $\mathcal{G}^{\infty
} $-microlocal hypoellipticity has been studied in \cite{HorObPil}. A
general interesting problem of \ $\left( M,\mathcal{A}\right) $-generalized
microlocal elliptic ultraregularity is to prove the following inclusion%
\begin{equation*}
WF_{\mathcal{A}}^{M}\left( f\right) \subset WF_{\mathcal{A}}^{M}\left(
P\left( x,D\right) f\right) \cup Char(P),\forall f\in \mathcal{G}\left(
\Omega \right) ,
\end{equation*}%
where $P\left( x,D\right) $ is a generalized linear partial differential
operator with $\mathcal{G}^{M,\mathcal{A}}\left( \Omega \right) $
coefficients and $Char(P)$\ is the set of generalized characteristic points
of $P(x,D).$
\end{remark}

\section{Generalized H\"{o}rmander's theorem}

We extend the generalized H\"{o}rmander's result on the wave front set of
the product, the proof follow the same steps as the proof of theorem 26 in
\cite{benbou2}. Let $f,g\in \mathcal{G}\left( \Omega \right) ,$ we define%
\begin{equation}
WF_{\mathcal{A}}^{M}\left( f\right) +WF_{\mathcal{A}}^{M}\left( g\right)
=\left\{ \left( x,\xi +\eta \right) :\left( x,\xi \right) \in WF_{\mathcal{A}%
}^{M}\left( f\right) ,\left( x,\eta \right) \in WF_{\mathcal{A}}^{M}\left(
g\right) \right\} ,  \label{4-1}
\end{equation}%
We recall the following fundamental lemma, see \cite{HorKun} for the proof.

\begin{lemma}
\label{ch4-lem2}Let $\sum_{1}$, $\sum_{2}$ be closed cones in $\mathbb{R}%
^{m}\backslash \left\{ 0\right\} ,$ such that $0\notin \sum_{1}+\sum_{2}$ ,
then

i) $\overline{\sum\nolimits_{1}+\sum\nolimits_{2}}^{\mathbb{R}^{m}/\left\{
0\right\} }=\left( \sum\nolimits_{1}+\sum\nolimits_{2}\right) \cup
\sum\nolimits_{1}\cup \sum\nolimits_{2}$

ii) For any open conic neighborhood $\Gamma $ of $\sum\nolimits_{1}+\sum%
\nolimits_{2}$ in $\mathbb{R}^{m}\backslash \left\{ 0\right\} ,$ one can
find open conic neighborhoods of $\Gamma _{1},$ $\Gamma _{2}$ in $\mathbb{R}%
^{m}\backslash \left\{ 0\right\} $ of, respectively, $\sum\nolimits_{1},\sum%
\nolimits_{2}$ , such that
\begin{equation*}
\Gamma _{1}+\Gamma _{2}\subset \Gamma
\end{equation*}
\end{lemma}

The principal result of this section is the following theorem.

\begin{theorem}
Let $f,g\in \mathcal{G}\left( \Omega \right) $ , such that $\forall x\in
\Omega ,$%
\begin{equation}
\left( x,0\right) \notin WF_{\mathcal{A}}^{M}\left( f\right) +WF_{\mathcal{A}%
}^{M}\left( g\right) ,  \label{ch4-4-1}
\end{equation}%
then
\begin{equation}
WF_{\mathcal{A}}^{M}\left( fg\right) \subseteq \left( WF_{\mathcal{A}%
}^{M}\left( f\right) +WF_{\mathcal{A}}^{M}\left( g\right) \right) \cup WF_{%
\mathcal{A}}^{M}\left( f\right) \cup WF_{\mathcal{A}}^{M}\left( g\right)
\label{ch4-4-2}
\end{equation}
\end{theorem}

\begin{proof}
Let $\left( x_{0},\xi _{0}\right) \notin \left( WF_{\mathcal{A}}^{M}\left(
f\right) +WF_{\mathcal{A}}^{M}\left( g\right) \right) \cup WF_{\mathcal{A}%
}^{M}\left( f\right) \cup WF_{\mathcal{A}}^{M}\left( g\right) ,$ then $%
\exists \phi \in D^{M}\left( \Omega \right) $, $\phi \left( x_{0}\right) =1,$
$\xi _{0}\notin \left( \sum_{\mathcal{A}}^{M}\left( \phi f\right) +\sum_{%
\mathcal{A}}^{M}\left( \phi g\right) \right) \cup \sum_{\mathcal{A}%
}^{M}\left( \phi f\right) \cup \sum_{\mathcal{A}}^{M}\left( \phi g\right) .$
From (\ref{ch4-4-1}) we have $0\notin \sum_{\mathcal{A}}^{M}\left( \phi
f\right) +\sum_{\mathcal{A}}^{M}\left( \phi g\right) $ then by lemma \ref%
{ch4-lem2} i), we have
\begin{equation*}
\xi _{0}\notin \left( \sum\nolimits_{\mathcal{A}}^{M}\left( \phi f\right)
+\sum\nolimits_{\mathcal{A}}^{M}\left( \phi g\right) \right) \cup
\sum\nolimits_{\mathcal{A}}^{M}\left( \phi f\right) \cup \sum\nolimits_{%
\mathcal{A}}^{M}\left( \phi g\right) =\overline{\sum\nolimits_{\mathcal{A}%
}^{M}\left( \phi f\right) +\sum\nolimits_{\mathcal{A}}^{M}\left( \phi
g\right) }^{\mathbb{R}^{m}\backslash \left\{ 0\right\} }
\end{equation*}%
Let $\Gamma _{0}$ be an open conic neighborhood of $\sum_{\mathcal{A}%
}^{M}\left( \phi f\right) +\sum_{\mathcal{A}}^{M}\left( \phi g\right) $ in $%
\mathbb{R}^{m}\backslash \left\{ 0\right\} $ such that $\xi _{0}\notin
\overline{\Gamma }_{0}$ then, from lemma \ref{ch4-lem2} ii), there exist
open cones $\Gamma _{1}$ and $\Gamma _{2}$ in $\mathbb{R}^{m}\backslash
\left\{ 0\right\} $ such that
\begin{equation*}
\sum\nolimits_{\mathcal{A}}^{M}\left( \phi f\right) \subset \Gamma _{1},%
\text{ }\sum\nolimits_{\mathcal{A}}^{M}\left( \phi g\right) \subset \Gamma
_{2}\text{ and }\Gamma _{1}+\Gamma _{2}\subset \Gamma _{0}
\end{equation*}%
Define $\Gamma =\mathbb{R}^{m}\backslash \overline{\Gamma }_{0},$ so
\begin{equation}
\Gamma \cap \Gamma _{2}=\emptyset \text{ and }\left( \Gamma -\Gamma
_{2}\right) \cap \Gamma _{1}=\emptyset  \label{ch4-4.3}
\end{equation}%
Let $\xi \in \Gamma $ and $\varepsilon \in \left] 0,1\right] $
\begin{eqnarray*}
\mathcal{F}\left( \phi f_{\varepsilon }\phi g_{\varepsilon }\right) \left(
\xi \right) &=&\left( \mathcal{F}\left( \phi f_{\varepsilon }\right) \ast
\mathcal{F}\left( \phi g_{\varepsilon }\right) \right) \left( \xi \right) \\
&=&\int_{\Gamma _{2}}\mathcal{F}\left( \phi f_{\varepsilon }\right) \left(
\xi -\eta \right) \mathcal{F}\left( \phi g_{\varepsilon }\right) \left( \eta
\right) d\eta +\int_{\Gamma _{2}^{c}}\mathcal{F}\left( \phi f_{\varepsilon
}\right) \left( \xi -\eta \right) \mathcal{F}\left( \phi g_{\varepsilon
}\right) \left( \eta \right) d\eta \\
&=&I_{1}\left( \xi \right) +I_{2}\left( \xi \right)
\end{eqnarray*}%
From (\ref{ch4-4.3}), $\exists a_{1}\geq 0,b_{1}\geq
0,k_{1}>0,c_{1}>0,\varepsilon _{1}\in \left] 0,1\right] ,$ such that $%
\forall \varepsilon \leq \varepsilon _{1},\forall \xi \in \Gamma _{2},$
\begin{equation*}
\left\vert \mathcal{F}\left( \phi f_{\varepsilon }\right) \left( \xi -\eta
\right) \right\vert \leq c_{1}\varepsilon ^{-b_{1}}\exp -\widetilde{M}\left(
k_{1}\varepsilon ^{a_{1}}\left\vert \xi \right\vert \right)
\end{equation*}%
we can show easily by the fact that $\left( \phi g_{\varepsilon }\right) \in
\mathcal{G}_{C}\left( \Omega \right) $ that $\forall a_{2}\geq 0,\forall
k_{2}>0,\exists b_{2}\geq 0,\exists c_{2}>0,,$ $\exists \varepsilon _{2}\in %
\left] 0,1\right] ,$ such that $\forall \varepsilon \leq \varepsilon _{2},$
\begin{equation*}
\left\vert \mathcal{F}\left( \phi g_{\varepsilon }\right) \left( \eta
\right) \right\vert \leq c_{2}\varepsilon ^{-b_{2}}\exp \widetilde{M}\left(
k_{2}\varepsilon ^{a_{2}}\left\vert \eta \right\vert \right) ,\forall \eta
\in \mathbb{R}^{n},
\end{equation*}%
Let $\gamma >0$ sufficiently small such that $\left\vert \xi -\eta
\right\vert \geq \gamma \left( \left\vert \xi \right\vert +\left\vert \eta
\right\vert \right) ,\forall \eta \in \Gamma _{2}.$ Hence for $\varepsilon
\leq \min \left( \varepsilon _{1},\varepsilon _{2}\right) ,$%
\begin{equation*}
\left\vert I_{1}\left( \xi \right) \right\vert \leq c_{1}c_{2}\varepsilon
^{-b_{1}-b_{2}}\int_{\Gamma _{2}}\exp \left( -\widetilde{M}\left(
k_{1}\varepsilon ^{a_{1}}\left\vert \xi -\eta \right\vert \right) +%
\widetilde{M}\left( k_{2}\varepsilon ^{a_{2}}\left\vert \eta \right\vert
\right) \right) d\eta
\end{equation*}%
from proposition \ref{pro2}, $\exists H>0,\exists A>0,\forall
t_{1}>0,\forall t_{2}>0,$
\begin{equation}
-\widetilde{M}\left( t_{1}+t_{2}\right) \leq -\widetilde{M}\left( \frac{t_{1}%
}{H}\right) -\widetilde{M}\left( \frac{t_{2}}{H}\right) +\ln A\text{ ,}
\end{equation}%
then
\begin{eqnarray*}
\left\vert I_{1}\left( \xi \right) \right\vert &\leq &c_{1}c_{2}\varepsilon
^{-b_{1}-b_{2}}\exp \left( -\widetilde{M}\left( \frac{k_{1}}{H}\varepsilon
^{a_{1}}\gamma \left\vert \xi \right\vert \right) \right) \\
&&\times \int_{\Gamma _{2}}\exp \left( -\widetilde{M}\left( \frac{k_{1}}{H}%
\varepsilon ^{a_{1}}\gamma \left\vert \eta \right\vert \right) +\widetilde{M}%
\left( k_{2}\varepsilon ^{a_{2}}\left\vert \eta \right\vert \right) \right)
d\eta \\
&\leq &c_{1}c_{2}\varepsilon ^{-b_{1}-b_{2}}\exp \left( -\widetilde{M}\left(
\frac{k_{1}}{H}\varepsilon ^{a_{1}}\gamma \left\vert \xi \right\vert \right)
\right) \\
&&\times \int_{\Gamma _{2}}\exp \left( -\widetilde{M}\left( \left( \frac{%
k_{1}}{H^{2}}\varepsilon ^{a_{1}}\gamma -k_{2}\varepsilon ^{a_{2}}\right)
\left\vert \eta \right\vert \right) \right) d\eta
\end{eqnarray*}%
take $k=\frac{\gamma k_{1}}{H}$ and $\frac{k_{1}}{H^{2}}\varepsilon
^{a_{1}}\gamma -k_{2}\varepsilon ^{a_{2}}>0,$ then $\exists b=b\left(
b1+b_{2},a_{1},a_{2},k_{1},k_{2},H\right) ,\exists c=c_{1}c_{2}$%
\begin{equation*}
\left\vert I_{1}\left( \xi \right) \right\vert \leq c\varepsilon ^{-b}\exp
\left( -\widetilde{M}\left( k\varepsilon ^{a_{1}}\left\vert \xi \right\vert
\right) \right)
\end{equation*}

Let $r>0,$
\begin{eqnarray*}
I_{2}\left( \xi \right) &=&\int_{\Gamma _{2}^{c}\cap \left\{ \left\vert \eta
\right\vert \leq r\left\vert \xi \right\vert \right\} }\mathcal{F}\left(
\phi f_{\varepsilon }\right) \left( \xi -\eta \right) \mathcal{F}\left( \phi
g_{\varepsilon }\right) \left( \eta \right) d\eta +\int_{\Gamma _{2}^{c}\cap
\left\{ \left\vert \eta \right\vert \geq r\left\vert \xi \right\vert
\right\} }\mathcal{F}\left( \phi f_{\varepsilon }\right) \left( \xi -\eta
\right) \mathcal{F}\left( \phi g_{\varepsilon }\right) \left( \eta \right)
d\eta \\
&=&I_{21}\left( \xi \right) +I_{22}\left( \xi \right)
\end{eqnarray*}%
Choose $r$ sufficiently small such that $\left\{ \left\vert \eta \right\vert
\leq r\left\vert \xi \right\vert \right\} \Longrightarrow \xi -\eta \notin
\Gamma _{1}.$ Then $\left\vert \xi -\eta \right\vert \geq \left( 1-r\right)
\left\vert \xi \right\vert \geq \left( 1-2r\right) \left\vert \xi
\right\vert +\left\vert \eta \right\vert $ , consequently $\exists
c>0,\exists a_{1},a_{2},b_{1},k_{1},k_{2}>0,\exists \varepsilon _{1}>0$ such
that $\forall \varepsilon \leq \varepsilon _{1},$
\begin{eqnarray*}
\left\vert I_{21}\left( \xi \right) \right\vert &\leq &c\varepsilon
^{-b}\int_{\Gamma _{2}}\exp \left( -\widetilde{M}\left( k_{1}\varepsilon
^{a_{1}}\left\vert \xi -\eta \right\vert \right) -\widetilde{M}\left(
k_{2}\varepsilon ^{a_{2}}\left\vert \eta \right\vert \right) \right) \\
&\leq &c\varepsilon ^{-b}\exp \left( -\widetilde{M}\left( k_{1}^{\prime
}\varepsilon ^{a_{1}}\left\vert \xi \right\vert \right) \right) \int \exp
\left( -\widetilde{M}\left( k_{1}\varepsilon ^{a_{1}}\left\vert \eta
\right\vert \right) -\widetilde{M}\left( k_{2}\varepsilon ^{a_{2}}\left\vert
\eta \right\vert \right) \right) d\eta \\
&\leq &c^{\prime }\varepsilon ^{-b^{\prime }}\exp \left( -\widetilde{M}%
\left( k_{1}^{\prime }\varepsilon ^{a_{1}}\left\vert \xi \right\vert \right)
\right)
\end{eqnarray*}%
If $\left\vert \eta \right\vert \geq r\left\vert \xi \right\vert ,$ we have $%
\left\vert \eta \right\vert \geq \dfrac{\left\vert \eta \right\vert
+r\left\vert \xi \right\vert }{2}$ , and then $\exists c>0,\exists
a_{1},b_{1},k_{1}>0,\forall a_{2},k_{2}>0,\exists b_{2}>0,\exists
\varepsilon _{2}>0$ such that $\forall \varepsilon \leq \varepsilon _{2}$%
\begin{eqnarray*}
\left\vert I_{21}\left( \xi \right) \right\vert &\leq &c\varepsilon
^{-b_{1}-b_{2}}\int_{\Gamma _{2}}\exp \left( \widetilde{M}\left(
k_{2}\varepsilon ^{a_{2}}\left\vert \xi -\eta \right\vert \right) -%
\widetilde{M}\left( k_{1}\varepsilon ^{a_{1}}\left\vert \eta \right\vert
\right) \right) d\eta \\
&\leq &c\varepsilon ^{-b_{1}-b_{2}}\int_{\Gamma _{2}}\exp \left( \widetilde{M%
}\left( k_{2}\varepsilon ^{a_{2}}\left\vert \xi -\eta \right\vert \right) -%
\widetilde{M}\left( \frac{k_{1}}{2}\varepsilon ^{a_{1}}\left\vert \eta
\right\vert +\frac{k_{1}r}{2}\varepsilon ^{a_{1}}\left\vert \xi \right\vert
\right) \right) d\eta \\
&\leq &c\varepsilon ^{-b_{1}-b_{2}}\exp \left( -\widetilde{M}\left( \frac{%
k_{1}r}{2H}\varepsilon ^{a_{1}}\left\vert \xi \right\vert \right) \right)
\int_{\Gamma _{2}}\exp \left( \widetilde{M}\left( k_{2}\varepsilon
^{a_{2}}\left\vert \xi -\eta \right\vert \right) -\widetilde{M}\left( \frac{%
k_{1}}{2H}\varepsilon ^{a_{1}}\left\vert \eta \right\vert \right) \right)
d\eta
\end{eqnarray*}%
if we take $k_{2},\frac{1}{a_{2}}$ sufficiently smalls, we obtain $\exists
a,b,c>0,\exists \varepsilon _{3}>0$, such that $\forall \varepsilon \leq
\varepsilon _{3}$
\begin{equation*}
\left\vert I_{21}\left( \xi \right) \right\vert \leq c\varepsilon ^{-b}\exp
\left( -\widetilde{M}\left( k\varepsilon ^{a}\left\vert \xi \right\vert
\right) \right) ,
\end{equation*}%
which finishes the proof.
\end{proof}

\end{document}